\documentclass[12pt]{amsart}
\usepackage{latexsym}
\usepackage{amsfonts}
\usepackage{graphicx}
\usepackage{amsfonts}
\usepackage{inputenc} 
\usepackage{comment}
\usepackage{amsmath,amscd}
\usepackage{color}
\usepackage[colorlinks=true, linkcolor=blue, citecolor=blue, urlcolor=blue]{hyperref}
\usepackage{epsfig}
\usepackage{amsthm,amsfonts}
\usepackage{amssymb,graphicx,color}
\usepackage[all]{xy} 
\usepackage{verbatim}
\usepackage{hyperref}
\usepackage{tikz-cd,mathtools}

\newcommand{\fin}{\hspace*{\fill}$\square$\vspace*{2mm}}

\makeatletter
\@namedef{subjclassname@1991}{$2020$ Mathematics Subject Classification}
\makeatother
 
\theoremstyle{plain}
\newtheorem{theorem}{Theorem}[section]
\newtheorem{lemma}[theorem]{Lemma}
\newtheorem{proposition}[theorem]{Proposition}

\theoremstyle{definition}
\newtheorem{definition}[theorem]{Definition}
\theoremstyle{remark}

\def\bC{{\mathbb C}}

\def\bN{{\mathbb N}}

\def\bR{{\mathbb R}}

\def\Si{{ \mathcal{S}_{\infty}}}

\def\Ker{{\rm Ker}}

\def\dist{{\rm dist}}

\def\Sing{{\rm Sing}}

\def\const.{{\rm const.}}

\title[A Thom Isotopy Theorem for nonproper semialgebraic maps]{A Thom Isotopy Theorem for nonproper semialgebraic maps}

\author{Luis Renato Gonçalves  Dias}
\address{Universidade Federal de Uberl\^andia, Faculdade de Matem\'atica, Av. Jo\~ao Naves de \'Avila 2121, 1F-153 - CEP: 38408-100, Uberl\^andia, Brazil.}
\email{lrgdias@ufu.br}

\author{Giovanny Snaider Barrera Ramos}
\address{Universidade de S\~ao Paulo, ICMC-USP, Av. Trab. São Carlense, 400-Centro,  CEP: 13566-590, S\~ao Carlos-SP, Brazil.}
\email{giovannybarreramos@usp.br}
 
\date {\today}



\begin{document}
	
\begin{abstract} 
We prove a version of the Thom Isotopy Theorem for nonproper semialgebraic maps $f\colon X\rightarrow \bR^m$, where $X \subset\bR^n$ is a semialgebraic set and $f$ is the restriction to $X$ of a smooth semialgebraic map $F\colon\bR^n\rightarrow \bR^m$.  
\end{abstract}

\maketitle
  
\section{Introduction}

Let $F\colon\bR^n \rightarrow \bR^m$ be a   semialgebraic map. We say that $F$ is  locally trivial at $t_0\in\bR^m$ if there exist a neighborhood $U$ of $t_0$ and  a diffeomorphism $h\colon F^{-1}(t_0)\times U\rightarrow F^{-1}(U)$ such that $F\circ h=pr_2$, where $pr_2\colon F^{-1}(t_0)\times U\rightarrow U$ denotes the second projection. The set of the points at which the  local triviality of $F$ fails is called the bifurcation set of $F$ and is denoted by $B(F)$.  It is well known that $B(F)$ is the union of the singular values set of $F$, denoted by $F(\Sing F)$,   and the set of bifurcation values at infinity of $F$. 

A full characterisation of $B(F)$ is a challenging open problem. In order to estimate $B(F)$, we can use conditions that control the asymptotic behaviour of the fibres of $F$. Rabier \cite{Ra} considered the asymptotic critical values set $K_\infty(F)$ defined by
\begin{align}\label{eq:kity}
K_\infty(F):=&\{t\in\bR^m\mid \exists (x_j)_{j\in \bN}\subset\bR^n, {\lim_{j\rightarrow\infty}\|x_j\|}=\infty\\ \nonumber 
&{\lim_{j\rightarrow\infty}F(x_j)}=t\text{ and }{\lim_{j\rightarrow\infty}\|x_j\|\nu(d_{x_j}F) }=0\},
\end{align}
where $\nu(d_xF):=\inf_{\|\varphi\|=1}\|d_xF^*(\varphi)\|$ and $d_xF^*$ denotes  the adjoint  of $d_xF$. From \cite{Ra}, see also \cite{Ga}, \cite{Je} and \cite{KOS}, it follows that  $ B(F)\subset F(\Sing F)\cup K_\infty(F)$.  In fact,   Rabier's result applies to $C^2$ maps $g\colon M\rightarrow N$, where $M$ and $N$ are Finsler Manifolds. 

In \cite{DRT}, the authors considered the following set: 
\begin{equation}\label{eq:Sity}
\mathcal{S}(F)=\{y\in\bR^m\mid\exists (x_j)_{j\in\bN}\subset M(F), \lim_{j\rightarrow \infty}\|x_j\|=\infty\text{ and }\lim_{j\rightarrow \infty} F(x_j)=y\},
\end{equation}
where $M(F)$ is the critical locus of the map  $(F,\rho):\bR^n \to \bR^{m+1}$, with $\rho(x_1, \ldots, x_n)=x_1^2 + \ldots +x_n^2$. From \cite{DRT}, it follows that $\mathcal{S}(F)\subset K_\infty(F)$ and $B(F)\subset F(\Sing F) \cup \mathcal{S}(F)$. We  call   $F(\Sing F) \cup \mathcal{S}(F)$ the  $\rho$ non-regular values set of $F$.  

We remark that the results presented above hold when we consider the restriction of $F$ to  a smooth semialgebraic set $X$.    

If $X$ is a singular set and $\mathcal{W}$ is a Whitney stratification of $X$, the Thom  Isotopy Theorem (see for instance \cite{GWPL}) ensures that every smooth map $f\colon X\to \bR^p$ is locally trivial on any $y\in \bR^p$, provided that the restriction of $f$ to any stratum of $\mathcal{W}$ is a submersion and the restriction of $f$ to the closure of any stratum is a proper map.

Recently, in the complex case, {\DH}inh and Jelonek \cite{DJ}  considered the singular case as follows. Let  $X\subset\bC^n$ be an affine variety and  let $F\colon\bC^n\rightarrow\bC^m$ be a polynomial map. Let $f\colon X\rightarrow\bC^m$ be the restriction of $F$ to $X$ and let  $\mathcal{W}=(X_\alpha)_{\alpha\in J}$ be a Whitney stratification of $X$. Motivated by \eqref{eq:kity}, the authors defined the  stratified generalized critical values set of $f$, denoted by  $\mathcal{K}(f,\mathcal{W} )$, and proved that   $B(f)\subset \mathcal{K}(f,\mathcal{W})$. In the real case, Anna Valette \cite{Va} considered stratified submersion $f\colon (X,\Sigma)\rightarrow Y$, where $X$ and $Y$  are semialgebraic sets, and proved that  $B(f) \subset 
 K_\Sigma(f)=\bigcup_{S\in \Sigma}K_\infty(f_{|_{S}})$. The trivialization obtained in \cite{Va} is a semialgebraic trivialization.   In particular, the results of \cite{DJ} and \cite{Va} can be seen as a extension of the  Thom Isotopy Theorem for nonproper maps.  

In this paper, we consider a  semialgebraic set $X\subset\bR^n$ and a Verdier $(w)$ stratification $\mathcal{W}=(X_\alpha)_{\alpha\in J}$ of $X$, see Section \ref{s:2} for details. Given a smooth semialgebraic map $F\colon\bR^n\rightarrow\bR^m$, we consider $f\colon X\rightarrow\bR^m$ the restriction of $F$ to $X$. Motivated by \eqref{eq:Sity}, we define the stratified $\rho$ non-regular values of $f$, denoted by $\mathfrak{S}(f,\mathcal{W})$, see Definition \ref{defconjmilnor}. Then, we prove that  $B(f)\subset \mathfrak{S}(f,\mathcal{W})$, see Theorem \ref{t:main}.  We also prove a trivialization theorem for $f$ outside a compact set in Theorem \ref{t:main-ity}.  

The proof of Theorem \ref{t:main-ity} is motivated by the arguments presented in  \cite{DJ}. We have constructed $m$ rugose vector fields on $X$. These are defined on the complement of a sphere of sufficiently large radius $R$.  They are tangent to $(n-1)$-dimensional spheres, and by integrating the $m$ vector fields we obtain the diffeomorphism trivializing $f$ outside a compact set. In the proof of Theorem \ref{t:main}, we defined $m$ rugose vector fields on the ball of radius $R$. We glue these $2m$ rugose fields to obtain $m$ rugose vector fields on $X$ which trivialize $f$. 

In Section \ref{s:2} we introduce the Verdier stratification, rugose vector fields and results for developing our objective.  In section \ref{s:3}, we proof our main results, namely Theorems \ref{t:main}  and \ref{t:main-ity}. 

 We point out that it follows from \cite{DRT} that $\mathfrak{S}(f,\mathcal{W})\subseteq \mathcal{K}(f,\mathcal{W})$. So, our result refines the results of \cite{DJ, Va}.
 
\section{Stratification and rugose vector fields}\label{s:2}

\subsection{Stratification}

\begin{definition}\label{d:stratification}
Let $X$ be a closed subset of $\bR^n$. A stratification of $X$ is a collection $\mathcal{W}=(X_\alpha)_{\alpha\in J}$ of pairwise disjoint connected smooth submanifolds of $\bR^n$ satisfying the following conditions
\begin{itemize}
    \item[i)] $X=\bigcup_{\alpha\in I}X_\alpha$.
    \item[ii)] $\mathcal{W}$ is locally finite. 
    \item[iii)] the frontier condition:  $X_\beta\cap\overline{X}_\alpha\neq\emptyset$ implies   $X_\beta\subset\overline{X}_\alpha$.    
\end{itemize}
Each ${X}_\alpha$ is called a stratum of $\mathcal{W}$. If, in addition $X$, $\overline{X}_\alpha$ and $\overline{X}_\alpha\setminus X_\alpha$ are semialgebraic sets, the stratification $\mathcal{W}$ is called an affine stratification of $X$. 
\end{definition}

\subsection{Whitney stratification}

Let $X$ be a closed subset of $\bR^n$ and $\mathcal{W}=(X_\alpha)_{\alpha\in J}$ a stratification of $X$. We start by recalling the conditions of Whitney, see \cite[Sections 18-19]{Wh} or \cite[Chapter 1]{GWPL}.   
 


Let $x,y\in\bR^n$. We denote by $\overline{xy}$ the line passing through the origin and  parallel to the line formed by the points $x$ and $y$. 

\begin{definition}\label{def Whitneyb}
Let $X_\alpha, X_\beta$ be strata of $\mathcal{W}$ such that $X_{\beta}\subset\overline{X}_\alpha$ and let $y \in X_\beta$. The pair $(X_\alpha,X_\beta)$ is Whitney $(b)$ regular at $y \in X_\beta$, provided that for any sequence $(x_n)_{n\in\mathbb{N}}\subset X_\alpha$ such that ${\lim_{n\rightarrow\infty}}x_n=y$ and ${\lim_{n\rightarrow\infty}}T_{x_n}X_\alpha=T$, and for any sequence $(y_n)_{n\in\bN}\subset X_\beta$ such that ${\lim_{n\rightarrow\infty}}y_n=y$ and ${\lim_{n\rightarrow\infty}}\overline{x_ny_n}=l$, we have $l\subset T$.

The pair $(X_\alpha,X_\beta)$ is Whitney $(b)$ regular if $(X_{\alpha},X_{\beta})$ is Whitney $(b)$ regular at any $y\in X_\beta$. The stratification $\mathcal{W}$ is a Whitney $(b)$ stratification  if $(X_\alpha,X_\beta)$ is Whitney $(b)$ regular for any strata $X_\alpha, X_\beta$ of $\mathcal{W}$ such that $X_{\beta}\subset\overline{X}_\alpha$. 
\end{definition}

\begin{definition}\label{defdis}
Let $V_1, V_2$ be vector subspaces of $\bR^n$. We define $$\delta(V_1, V_2)={\sup_{{a\in V_1}; \text{ }{\|a\|=1}}}\|a-\pi_{V_2}(a)\|,$$
where $\pi_{V_2}$ denotes the orthogonal projection from $\bR^n$ to $V_2$.
\end{definition}

\begin{definition}\label{def verdierw}
Let $X_\alpha, X_\beta$ be strata of $\mathcal{W}$ such that $X_{\beta}\subset\overline{X}_\alpha$ and let $y\in X_\beta$. The pair $(X_\alpha,X_\beta)$ is Verdier $(w)$ regular at $y\in X_{\beta}$, provided that there exist a constant $C>0$ and a neighborhood $U$ of $y$ in $\bR^n$ such that for any $y'\in U\cap X_\beta$ and any $x\in U\cap X_\alpha$, we have 
\begin{equation}
\delta(T_{y'}X_\beta,T_{x}X_\alpha)\leq C\|y' -x\|.    
\end{equation}  
The pair $(X_\alpha,X_\beta)$ is Verdier $(w)$ regular if $(X_{\alpha},X_\beta)$  is Verdier $(w)$ regular at any $y\in X_\beta$. The stratification $\mathcal{W}$ is a Verdier $(w)$ stratification  if $(X_\alpha,X_\beta)$ is  Verdier $(w)$ regular  for any strata $X_\alpha, X_\beta$ of $\mathcal{W}$ such that $X_{\beta}\subset\overline{X}_\alpha$. 
\end{definition}

We have that any semialgebraic set $X \subset \bR^n$ admits a Whitney $(b)$ stratification whose strata are semialgebraic sets, see  
 \cite{Leloi} or Chapter I of \cite{GWPL}. Also, any semialgebraic set  $X \subset \bR^n$ admits a Verdier $(w)$ stratification $\mathcal{W}$ whose strata are semialgebraic sets, see for instance \cite{Leloi}.  

 Let $X$ be a semialgebraic set and let $g\colon X\rightarrow \bR$ be a semialgebraic function.   Let $\mathcal{W}$ be an affine stratification of $X$. For any  $X_\beta\in \mathcal{W}$, we denote by $g_{|_{X_\beta}}$ the restriction of $g$ to $X_\beta$ and we set $T_{x,g}:=kerD(g_{|_{X_\beta}})(x)$. 
 
 \begin{definition}
Let $X\subset \bR^n$ be a semialgebraic set and let $g\colon X\rightarrow \bR$ be a semialgebraic function. Let $\mathcal{W}$ be an affine stratification of $X$ such that for any stratum $X_\gamma \in \mathcal{W}$, the restriction $g_{|_{X_\gamma}}$ is of constant rank. Let $X_\alpha,X_\beta\in \mathcal{W}$ be such that $X_\beta\subset \overline{X}_\alpha$. We say that $(X_\alpha,X_\beta)$ satisfies the strict Thom condition $(w_f)$ at $y \in X_{\beta}$ if there exist a constant $C>0$ and a neighborhood $U$ of $y$ in $\bR^n$  such that
\begin{equation*}
\delta(T_{x,f}, T_{y',f}) \leq C\|x-y'\|,    
\end{equation*}
for any  $x\in X_\alpha\cap U \mbox{ and any } y'\in X_\beta\cap U$.

The pair $(X_\alpha,X_\beta)$ is $(w_f)$ regular if $(X_{\alpha},X_\beta)$  is $(w_f)$ regular at any $y\in X_\beta$.  
\end{definition}

We have the following result, see   \cite[Section 8]{PP} or \cite[Section 1]{HMS}.  

\begin{proposition}\label{teowf}  
Let $X\subset \bR^n$ be a semialgebraic set and let $g\colon X\rightarrow \bR$ be a semialgebraic function. Let $\mathcal{W}$ be an affine Verdier $(w)$ stratification of $X$ such that for any stratum $X_\gamma \in \mathcal{W}$, the restriction $g_{|_{X_\gamma}}$ is of constant rank. Let $X_\alpha, X_\beta\in \mathcal{W}$ be strata of $X$ such that $X_\beta\subset \overline{X}_\alpha$.  If the restriction  $g_{|_{X_\beta}}$ is a submersion, then $(X_\alpha, X_\beta)$ is $(w_f)$ regular. 
\fin
\end{proposition}

\subsection{Rugose vector fields}

\begin{definition}\label{defrugos}
Let $X\subset\bR^n$ be an affine variety and let $\mathcal{W}=(X_{\alpha})_{\alpha\in J}$ be an  affine  stratification of $X$. A function $\varphi:X\rightarrow\bR$ is called a rugose function if the following conditions are satisfied
\begin{enumerate}
  \item The restriction $\varphi_{|_{X_\alpha}}$ to any stratum $X_{\alpha}$ is a smooth function.
    \item For any stratum $X_\alpha$, and for any $x\in X_\alpha$, there is a neighborhood $U$ of $x$ in $\bR^{n}$ and a constant $C>0$ such that 
    $$|\varphi(y)-\varphi(x')|\leq C\|y-x'\|,$$ 
  for any $y\in X\cap U$ and any $x'\in X_\alpha\cap U$. 
\end{enumerate} 

A map $\varphi:X\rightarrow \bR^m$ is called a rugose map if its components functions are rugose functions. A vector field $v$ on $X$ is called a rugose vector field if $v$ is a rugose map and,  for any $x\in X$, the vector $v(x)$ is tangent to the stratum containing $x$.     
\end{definition}

We have the following result, see Proposition 4.3 of \cite{Ve}:  

\begin{proposition}\label{prop:ext}
Let $X\subset\bR^n$ be a semialgebraic set and  let $\mathcal{W}=(X_\alpha)_{\alpha\in J}$ be an affine Verdier $(w)$ stratification of $X$. Let $X_\alpha\in \mathcal{W}$ and let $B\subset \overline{X}_\alpha$ a closed union of strata.  Let $\eta:B\rightarrow \bR^n$ be a rugose map on $B$. Then there exists a rugose  map  $\bar{\eta}: (B\cup X_\alpha)\rightarrow \bR^n$ on  $(B\cup X_\alpha)$ that extends $\eta$, that is $\bar{\eta}|_{B}=\eta$.
\fin
\end{proposition}

Under the above conditions, Verdier \cite[Corollary 4.4]{Ve} proved that if $\eta$ is a  rugose vector field on  $B$, then it  extends to a rugose vector field on  $A$. 

In the next theorem, we consider $\rho:\bR^n\to \bR$ defined by $\rho(x_1, \ldots, x_n)=x_1^2 +\ldots +x_n^2$. We prove that if $\eta$ is a rugose vector field on $B$ and is tangent to the $(n-1)$-dimensional sphere of radius $\|x\|$ for any $x\in B$, then $\eta$ extends to a rugose vector field on  $A$ and to the $(n-1)$-dimensional sphere of radius $\|y\|$ for any $y\in A$. 

\begin{theorem}\label{VTB}
Let $X\subset\bR^n$ be a closed semialgebraic set and let  $\mathcal{W}=(X_\alpha)_{\alpha\in J}$ be an affine Verdier $(w)$ stratification of $X$. Let $ {X}_{\alpha}\in \mathcal{W}$ and let $B \subset \overline{X}_{\alpha}$ be a closed union of strata. Suppose that the restriction of $\rho$  to any stratum of $B$ is a submersion.  Let $\eta:B\rightarrow \bR^n$ be a rugose vector field on  $B$ tangent to the $(n-1)$-dimensional sphere of radius $\|x\|$ for any $x\in B$. Then $\eta$ extends to a rugose vector field on $A:=B\cup X_{\alpha}$ tangent to  the $(n-1)$-dimensional sphere of radius $\|y\|$ for any  $y\in A$.
\end{theorem}
\begin{proof}   By  Proposition \ref{prop:ext}, there is a rugose map $ \bar{\eta} : X\rightarrow \bR^n$ that $\bar{\eta}$ is a extension of $\eta$. For any $y\in X_\alpha$, let $\hat{\eta}$ be the extension of $\eta$ obtained by projecting the vector $\bar{\eta}(y)$ on  $T_yX_\alpha\cap T_yS^{n-1}$. We will show that  $\hat{\eta}:X\rightarrow \bR^n$ is a rugose vector field. 

Since $\bar{\eta}$ is a rugose vector field, it follows that  if $x\in B$ then there exist  $C>0$ and a neighborhood $G_{x}$ of $x$ such that for any $x'\in B\cap G_{x}$ and  any $y\in A\cap G_{x}$, we have:  
\begin{equation}\label{eq:r1}
    \|\bar{\eta}(y)-\bar{\eta}(x')\|\leq C\|y - x'\|.
\end{equation}

Suppose that ${\hat{\eta}}$ is not a rugose vector field on $x$. Then  there exist sequences $(y_k)_{k\in \bN}\subset X_{\alpha}\cap G_{x}$ and $(x_k')_{k\in \bN}\subset B\cap G_{x}$ such that, $y_k\rightarrow x, \text{ }x_k'\rightarrow x$ and $$\|{\hat{\eta}}(y_k)-{\hat{\eta}}(x_k')\|\geq (k+C)\|x_k'-y_k\|.$$
Thus, 
\begin{equation}\label{eq:r2}
\frac{\|\hat{\eta}(y_k)-\bar{\eta}(y_k)\|}{\|x_k'-y_k\|}+\frac{\|\bar{\eta}(y_k)-{\hat{\eta}}(x_k')\|}{\|x_k'-y_k\|}\geq (k+C),   
\end{equation}
Since $\hat{\eta}$ is an extension of $\bar{\eta}$ and $x'_k \in B$,  it follows that $\hat{\eta}(x'_k)=\bar{\eta}(x'_k)$. Then, by \eqref{eq:r1} and \eqref{eq:r2}, we have: 
 \begin{equation}\label{eq campVchap1}
\frac{\|\hat{\eta}(y_k)-\bar{\eta}(y_k)\|}{\|x_k'-y_k\|}\geq k. 
\end{equation}

Then   $ \hat{\eta}(y_k)\neq \bar{\eta}(y_k)$ and, since  $\hat{\eta}(y_k)$ is  the projection of  $\bar{\eta}(y_k)$ on $T_yX_\alpha\cap T_yS^{n-1}$, we have $\bar{\eta}(y_k)\neq 0$. It follows that 
\begin{equation*}
\frac{\|\hat{\eta}(y_k)-\bar{\eta}(y_k)\|}{\|\bar{\eta}(y_k)\|\|x_k'-y_k\|}\geq \frac{k}{\|\bar{\eta}(y_k)\|}.   
\end{equation*}

Since $\|\hat{\eta}(y_k)-\bar{\eta}(y_k)\|=dist(\bar{\eta}(y_k),T_{y_k}X_\alpha\cap T_{y_k}S^{n-1})$ and for every linear space $T$ holds $\dist(\lambda v,T)= |\lambda| \dist(v,T)$, for any $\lambda \in \bR^*$, it follows that  

\begin{equation*}
\frac{\dist\left(\frac{\bar{\eta}(y_k)}{\|\bar{\eta}(y_k)\|},T_{y_k}X_\alpha\cap T_{y_k}S^{n-1}\right)}{\|x_k'-y_k\|}=\frac{\dist(\bar{\eta}(y_k),T_{y_k}X_\alpha\cap T_{y_k}S^{n-1})}{\|\bar{\eta}(y_k)\|\|x_k'-y_k\|}\geq \frac{k}{\|\bar{\eta}(y_k)\|}.   
\end{equation*}

Since the $(w)$ condition  is a local property and invariant under $C^2$ local diffeomorphisms, we can suppose that  $B$ is an open subset of $\bR^k\subset\bR^k\times \bR^{n-k}$. Thus, $T_{x}B=\bR^{k}$ for all $x\in B$. Then: 
\begin{equation}\label{eq45}
\frac{\delta(\bR^k\cap T_{x_k'}S^{n-1},T_{y_k}X_\alpha\cap T_{y_k}S^{n-1})}{\|z_k-y_k\|}\geq \frac{\dist\left(\frac{\bar{\eta}(y_k)}{\|\bar{\eta}(y_k)\|},T_{y_k}X_\alpha\cap T_{y_k}S^{n-1}\right)}{\|x_k'-y_k\|}\geq \frac{k}{\|\bar{\eta}(y_k)\|}\rightarrow\infty   
\end{equation}
when $k\rightarrow\infty$. Since the pair $(X_\alpha,B)$ is Verdier regular e $\rho_{|_{B}}$ is a submersion, then by Proposition \ref{teowf} we have that the pair $(X_\alpha,B)$ is $(w_\rho)$ regular, which implies that the expression on the left side of the equation \eqref{eq45} is bounded. This is a contradiction.
\end{proof}

The next result was originally proved  in \cite[Proposition 4.6]{Ve}   for real analytic sets  $X$ with a Whitney stratification. By similar argument  as in \cite{Ve}, we have: 

\begin{proposition}\label{p:lift}
Let $X\subset\bR^n$ be a semialgebraic set and $\mathcal{W}=(X_\alpha)_{\alpha\in J}$ be a Verdier $(w)$ stratification of $X$. Let $f\colon \bR^n\rightarrow\bR^m$ be  a map such that the restriction of $f$ on any stratum of $\mathcal{W}$ is a $C^{\infty}$ submersion. Then, any  rugose vector field on $\bR^m$ can be lifted to a rugose vector field on $X$.  \fin
\end{proposition}

\section{Main theorem}\label{s:3}

\subsection{Stratified Milnor set}

Let $X\subset\bR^n$ be a semialgebraic set and let $\mathcal{W}=(X_{\alpha})_{\alpha\in J}$ be an affine stratification of $X$, see Definition \ref{d:stratification}.

Let  $f\colon X\rightarrow\bR^m$ be a semialgebraic map, where $\dim X \geq m$.   Let  $\rho:\bR^n\rightarrow\bR$ be the function given by $\rho(x_1, \ldots, x_n)=x_1^2 + \ldots +x^2_n$. 

For any stratum $X_{\alpha}$ of $\mathcal{W}$, we denote by  $\Sing f_{|_{X_{\alpha}}}$  the critical locus of $f_{|_{X_{\alpha}}}$ and let $f(\Sing f_{|_{X_{\alpha}}})$ be the critical values set of $f_{|_{X_{\alpha}}}$. We denote by $ M(f_{|_{X_\alpha}})$ the critical locus of  $(f,\rho):  X_{\alpha} \to \bR^{m+1}$ and we define: 
 \begin{equation}
     \Si(f_{|_{X_\alpha}})=\{y\in\bR^m\mid\exists (x_j)_{j\in\bN}\subset M(f|_{X_\alpha}), x_j\rightarrow\infty\text{ and }\lim_{j\rightarrow \infty} f(x_j)=y\}, 
\end{equation}
where $x_j \rightarrow \infty$ means that the sequence $(x_j)_{j\in\bN}\subset M(f|_{X_\alpha})$ has no convergent subsequences in $M(f|_{X_\alpha})$.  
 
\begin{definition}\label{defconjmilnor}
 We call $M(f, \mathcal{W}):=\cup_{\alpha\in J} M(f_{|_{X\alpha}})$ the stratified  Milnor set of $f$ and we call  \begin{equation}\label{defS}
 \mathfrak{S}(f,\mathcal{W}):=\bigcup_{\alpha\in J }(f(\Sing f_{|_{X_{\alpha}}})\cup \Si(f_{|_{X_\alpha}})). 
 \end{equation}
 the stratified  $\rho$ non-regular values set of $f$. 
\end{definition} 
 
\begin{proposition}\label{p:closed}
Let $X\subset\bR^n$ be a semialgebraic set and let $\mathcal{W}=(X_\alpha)_{\alpha\in J}$ be an affine  stratification. Let $f\colon X\rightarrow\bR^m$ be a semialgebraic map.   Then, for any $\alpha\in J$, the sets  $\Si(f_{|_{X_\alpha}})$ and $(f(\Sing f_{|_{X_{\alpha}}})\cup \Si(f_{|_{X_\alpha}}))$  are closed.
\end{proposition}
\begin{proof}
Let $y\in \overline{\Si(f_{|_{X_\alpha}})}$ and let $(y^{k})_{k\in \bN}\subset \Si(f_{|_{X_\alpha}})$ be a sequence such that $\lim_{k\to \infty} y^k= y$. By definition of $\Si(f_{|_{X_\alpha}})$, for every $k\in \bN$, there exists a sequence $(x^{k}_{j})_{j\in\bN}\subset M(f|_{X_{\alpha}})$, such that 
$x^{k}_{j}\rightarrow\infty$ and ${\lim_{j\rightarrow \infty}}f(x^{k}_{j})=y^k$. Up to a  subsequence of $(y^{k})_{k\in \bN}$, we may suppose the following cases: 

\vspace{.2cm}

\noindent Case 1: $\lim_{j\to\infty}\|x^{k}_{j}\|=\infty$,    $k \in \bN$.

In this case, for each $k \in \bN$, there is $j(k)\in \mathbb{N}$ such that if $j\geq j(k)$, then $\|x^{k}_{j}\|>k$ and  $\|f(x^{k}_{j})-y^{k}\|<1/k$. Setting  $z^k=x^{k}_{j(k)}$, we have a sequence $\{z^k\}_{k\in\bN}\subset M(f|_{X_\alpha})$, such that ${\lim_{k\rightarrow \infty}}\|z^{k}\|=\infty$ and ${\lim_{k\rightarrow \infty}}f(z^k)=y$. Therefore, in this case, we have $y\in \Si(f_{|_{X_\alpha}})$.
 
\vspace{.2cm}

\noindent Case 2: ${\lim_{j\rightarrow \infty}}x^{k}_{j}= x^k\in \overline{X}_{\alpha}\setminus X_{\alpha}$. 

Thus, we have $f(x^k)=y^k$ since $f$ is a continuous map. For this case, up to a  subsequence of $(x^{k})_{k\in \bN}$, we may suppose the following subcases:

\vspace{.2cm}

i-) ${\lim_{k\rightarrow \infty}}\|x^{k}\|=\infty$. Then, for each $k \in \bN$, there is $j(k)\in \mathbb{N}$ such that if $j\geq j(k)$ then $\|x^{k}_{j}\|>k$, $\|x^{k}_{j}-x^{k}\|<1/k$ and $\|f(x^{k}_{j})-y^{k}\|<1/k$. Setting  $z^k=x^{k}_{j(k)}$, we have a sequence $\{z^k\}_{k\in\bN}\subset M(f|_{X_\alpha})$ such that ${\lim_{k\rightarrow \infty}}\|z^{k}\|=\infty$ and ${\lim_{k\rightarrow \infty}}f(z^k)=y$. Therefore,  $y\in \Si(f_{|_{X_\alpha}})$.

ii-)  ${\lim_{k\rightarrow \infty}}x^{k}= x\in \overline{X}_{\alpha}\setminus X_{\alpha}$. Then, for each $k \in \bN$, there is $j(k)\in \mathbb{N}$ such that if $j\geq j(k)$ then $\|x^{k}_{j}-x^{k}\|<1/k$ and $\|f(x^{k}_{j})-y^{k}\|<1/k$. Setting  $z^k=x^{k}_{j(k)}$, we have a sequence $\{z^k\}_{k\in\bN}\subset M(f|_{X_\alpha})$, such that ${\lim_{k\rightarrow \infty}}z^{k}=x$ and ${\lim_{k\rightarrow \infty}}f(z^k)=y$. Hence, $y\in \Si(f_{|_{X_\alpha}})$. 

This proves that $\Si(f_{|_{X_\alpha}})$ is a closed set. 

Now, let $t_0\in \overline{(f(\Sing f_{|_{X_{\alpha}}})\cup \Si(f_{|_{X_\alpha}}))}$. We may assume that $t_0\in \overline{f(\Sing f_{|_{X_{\alpha}}})}$ since $\Si(f_{|_{X_\alpha}})$ is a closed set. Then, there  exists a sequence $(x_j)_{j\in\bN}\subset \Sing f_{|_{X_{\alpha}}}$ such that ${\lim_{j\rightarrow \infty}}f(x_j)=t_0$. We have the following cases: 

 \noindent Case 1:   $x_j\rightarrow\infty$. Then, $t_0\in \Si(f_{|_{X_\alpha}})$ since $\Sing f_{|_{X_{\alpha}}} \subset M(f_{|_{X_{\alpha}}})$.

\noindent Case 2: ${\lim_{j\rightarrow \infty}}x_j=x_0 \in X_{\alpha}$. It follows that  $x_0\in \Sing f_{|_{X_{\alpha}}}$ since $\Sing f_{|_{X_{\alpha}}}$ is a closed set in $X_{\alpha}$. Then,   $f(x_0)=t_0\in f(\Sing f_{|_{X_{\alpha}}})$. 

Therefore, $t_0\in (f(\Sing f_{|_{X_{\alpha}}})\cup \Si(f_{|_{X_\alpha}}))$. This proves that $(f(\Sing f_{|_{X_{\alpha}}})\cup \Si(f_{|_{X_\alpha}}))$ is a closed set.
\end{proof}

\subsection{Main theorem}

In the following we denote by $B_{R}$ the  open ball centered at $0\in \bR^n$ with radius $R$. 

In this section we prove our main results: 

\begin{theorem}\label{t:main} Let $X\subset\bR^n$ be a semialgebraic set and let $\mathcal{W}=(X_{\alpha})_{\alpha\in J}$ be an affine Verdier $(w)$ stratification of $X$. Let  $f=(f_1,\dots, f_m)\colon X\rightarrow \mathbb{R}^m$ be the restriction of a smooth semialgebraic map $F:\bR^n \to \bR^m$ to $X$.  Then    $B(f) \subset \mathfrak{S}(f,\mathcal{W})$.  
\end{theorem}

First, we prove the following theorem:

 \begin{theorem}\label{t:main-ity} Let $X\subset\bR^n$ be a semialgebraic set and let $\mathcal{W}=(X_{\alpha})_{\alpha\in J}$ be an affine Verdier $(w)$ stratification of $X$. Let  $f=(f_1,\dots, f_m)\colon X\rightarrow \mathbb{R}^m$ be the restriction of a smooth semialgebraic map $F:\bR^n \to \bR^m$ to $X$.   Let $z\notin \mathfrak{S}(f,\mathcal{W})$ and $B$ an open ball such that $z\in B$ and $B\cap \mathfrak{S}(f,\mathcal{W})=\emptyset$. Then, there  is $R>0$ such that $f_{|_{X\cap B_R^c}}$ is a topologically trivial fibration over $B$. 
 \end{theorem}
\begin{proof}
Let $z\in\bR^m\setminus \mathfrak{S}(f,\mathcal{W})$. It follows by Proposition \ref{p:closed} that there exists an open ball $B$ such that $z\in B$ and    $\overline{B} \subset \bR^m\setminus \mathfrak{S}(f,\mathcal{W})$. After a linear change of coordinates, we may assume that $z$ is the origin $0$ and $B=(-1,1)^m$, that is, $B$ is the cartesian product of the open interval $(-1,1)$ $m$-times.

It follows by Definition \ref{defconjmilnor} that there exist   $R>0$ such that 
\begin{eqnarray}\label{eq:M=empty}
 (\bR^n\setminus \overline{B}_{R} ) \cap  X_{\beta}\cap f^{-1}(B) \cap M(f, \mathcal{W})  = \emptyset,
\end{eqnarray} 
for any stratum $X_{\beta}$ of $\mathcal{W}$. 

We set $U_{\beta}:= (\bR^n\setminus \overline{B}_{R} ) \cap X_{\beta}$ and $U:=\cup_{\beta \in J} U_{\beta}$.  

Then, for any $x\in f^{-1}(B) \cap  U$, the restriction $(f,\rho) : X_{\beta} \to \bR^{m+1}$ is a submersion at $x$, where $X_{\beta}$ is the stratum of $\mathcal{W}$ containing $x$. It follows that the restriction of the linear map \[ d_x(f_{|_{X_{\beta}}}): T_x X_{\beta} \to \bR^m\]  to $T_xX_{\beta} \cap T_xS^{n-1}$ is a surjective linear map, i.e. $d_x(f_{|_{X_{\beta}}}) (T_xX_{\beta} \cap T_xS^{n-1}) = \bR^m$. Then,  there exist  $m$ smooth vector fields $\{V_{1_\beta}, \dots, V_{m_\beta}\}$ on  $X_{\beta}$ satisfying  the following two conditions:  

 \begin{enumerate}
\item $\{V_{1_{\beta}}(x), \dots, V_{m_{\beta}}(x)\}$ is a basis of 
\[ Z_\beta =\{ w \in T_xX_{\beta} \cap T_xS^{n-1} \mid \langle w,u \rangle =0, \mbox{ for all } u\in (\Ker d_{x}(f_{|_{X_\beta}}))\cap  T_xS^{n-1} \}.\]

\item $d_{x}(f_{|_{X_\beta}})(V_{i_{\beta}}(x))=e_i, \mbox{ and } \langle \nabla\rho(x), V_{i_{\beta}}(x)\rangle =0,$ for any $i=1, \dots, m$. 
\end{enumerate}   
 
\vspace{.2cm}
  
We point out that the vector fields on  $f^{-1}(B) \cap U$ coinciding with $V_{1_{\beta}}, \ldots, V_{m_{\beta}}$ on  each stratum $  X_{\beta}$  are  not necessarily rugose vector fields. In what follows, we use these vector fields to produce  $m$ rugose vector fields on $f^{-1}(B) \cap U$. 

For $0 \leq d\leq \dim X$, let $I_d:=\{\beta\in J \mid 0 \leq \dim X_\beta\leq d\}$ and $U_d:=\cup_{\beta\in I_d}U_{\beta} \cap f^{-1}(B)$. Then, 
\[ \cup_{0\leq d \leq \dim X} U_d = f^{-1}(B) \cap U.\]

We define the $m$ rugose vector fields on $f^{-1}(B) \cap U$ inductively as follows: 

\vspace{.2cm}

\noindent {\bf First step}. We have $U_d= \emptyset$, for $0 \leq d \leq m$. 

In fact, for $0 \leq d <m$ and $\alpha \in I_d$,  we have  $f(X_{\alpha}) \subset f(\Sing f_{|X_{\alpha}})$. Since $B \cap f(\Sing f_{|X_{\alpha}})=\emptyset$, it follows that $U_d=\emptyset$, for $0 \leq d <m$.  

For $d=m$ and $\beta \in I_m$, we have $X_\beta \subset M(f_{|X_\beta})$. It follows by \eqref{eq:M=empty} and the definition of $U_\beta$ that $U_m = \emptyset$. 

\vspace{.2cm}

\noindent{\bf Second step}. For $d=m+1$ and $\beta \in I_{m+1}$, we define $V_i^{m+1}(x)=V_{i_\beta}(x)$, for any $x\in X_\beta$ and $i=1, \ldots, m$. Since each $V_{i_\beta}(x)$ is a smooth vector field on $X_\beta$, it follows that $V_1^{m+1}, \ldots, V_m^{m+1}$ are  rugose vector fields on $U_{m+1}$ satisfying  $df_x(V_i^{m}(x))=e_i$ and $\langle \nabla \rho(x), V_i^{m}(x) \rangle =0$, for any $x\in U_m$.

\noindent{\bf Third step}. Let $m+1 < d \leq \dim X$ and assume that we have constructed $m$ rugose vector fields  $V_1^{d}, \ldots, V_m^d$ on $U_d$ such that 
\[ d_x f(V_i^d(x))=e_i \mbox{ and } \langle \nabla \rho(x), V_i^d(x) \rangle =0, \mbox{ for any $x\in U_d$}.\] 

We will construct $V_1^{d+1}, \ldots, V_m^{d+1}$ separately on each stratum $X_\alpha$, with $\alpha\in I_{d+1}\setminus I_d$. Thus, we may assume without loss of generality that $I_{d+1}\setminus I_{d}=\{\alpha\}$. 

We fix $i\in \{ 1, \ldots, m \}$. It follows by Theorem \ref{VTB}  that there exists a rugose vector field $\widehat{V}_i^{d+1}$ on $U_{d+1}$, which  extends $V_i^{d}$, such that $\langle \nabla \rho(x), 
\widehat V_i^{d+1}(x)\rangle =0$, for any $x\in U_{d+1}.$  

Since  $\{V_{1_{\alpha}}(x), \dots, V_{m_{\alpha}}(x)\}$ is a basis of 
\[  Z_\alpha =\{ w \in T_xX_{\alpha } \cap T_xS^{n-1} \mid \langle w,u \rangle =0, \mbox{ for all } u\in (\Ker d_{x}(f_{|_{X_\alpha}}))\cap  T_xS^{n-1} \},\] 
we can write for each $x\ in U_{d+1} $ the following: 
\begin{equation}
\widehat{V}_i^{d+1}(x)=\sum_{j=1}^{m}a_j(x)V_{j_\alpha}(x)+ G_{\alpha}(x),
\end{equation}
where $G_{\alpha}(x) \in \{ z\in (T_xX_\alpha\cap T_xS^{n-1}) \mid \langle z,   V_{j_\alpha}(x) \rangle =0, \mbox{ for } j=1, \ldots, m\}$. We define ${V}_i^{d+1}$ on $U_{d+1}$ as follows: 
\begin{equation}\label{eq:defV}
{V}_i^{d+1}(x) := \left\lbrace
\begin{array}{ll}
V_{i_\alpha}(x)+ G_{\alpha}(x), \textup{  if } x\in U_{d+1}\setminus U_{d},\\
V_i^d(x), \textup{ if } x\in U_d.
\end{array}
\right.
\end{equation}
Then, we have:

\begin{lemma}\label{l:rug-mod}
 ${V}_i^{d+1}$ is a rugose vector field on $U_{d+1}$ that satisfies,  for any    $x\in U_{d+1}$, the following conditions:   
\begin{enumerate}
    \item $\langle \nabla \rho(x),  {V}_i^{d+1}(x)\rangle=0$.  

    \item $d_x f({V}_i^{d+1}(x))=e_i$. 
\end{enumerate} 
\end{lemma}
\begin{proof}
 For $x'\in U_d$ and $y\in U_{d+1}\setminus U_d$, we can write:  
 \begin{equation}\label{eq:v1}
  V_i^{d}(x')=\sum_{j=1}^{m}b_j(x',y)V_{j_\alpha}(y)+G(x', y)+ P(x', y) + S(x',y),  
 \end{equation}
where $G(x', y) \in \{ z\in (T_yX_\alpha\cap T_yS^{n-1}) \mid \langle z,   V_{j_\alpha}(y) \rangle =0, \mbox{ for } j=1, \ldots, m\}$, $P(x', y) $ is in  the complementary space of $T_yX_\alpha\cap T_yS^{n-1}$ in $T_yX_\alpha$ and $S(x',y)\in (T_yX_\alpha)^{\perp} \subset \bR^n$. 

Since $\widehat{V}_i^{d+1}$ is a rugose vector field, it follows that for each $x\in U_d$, there exist a constant $C>0$ and a neighborhood $W_x$ of $x$ such that for any $y\in W_x\cap X_\alpha$ and any $x'\in W_x\cap X_\beta$, where $X_\beta$ is a stratum  containing $x$,  we have 
\begin{equation}\label{eq:v2}
    \|\widehat{V}_i^{d+1}(y)-\widehat{V}_i^{d+1}(x')\|< C\|y-x'\|.
\end{equation}  

By \eqref{eq:defV}, \eqref{eq:v1} and \eqref{eq:v2} we have:
\begin{align*}
\|\sum_{j=1}^{m}(a_j(y)-b_j(x',y))V_{j_\alpha}(y)\| & \leq C\|x'-y\|,\\
\| G_{\alpha}(y)-G(x',y)  \| & \leq C\|x'-y\|,  \\ 
 \|P(x',y) \| & \leq C\|x'-y\|, \\ 
\|S(x',y)\| & \leq C\|x'-y\|.
\end{align*}

Shrinking $W_x$ and increasing $C$ if for necessary, we can suppose that $dF:X\rightarrow \mathcal{L}(\bR^n,\bR^m)$ defined by $dF(x)=d_xF$ is Lipschitz on $W_x$. Then, 
\[ \|d_yF-d_{x'}F\|< C\|y-x'\|.\] 

In particular, it follows that 
\[ \|d_yF((V_i^d(x'))- d_{x'}F((V_i^d(x'))\|< C\|y-x'\|\|V_i^d(x')\|.\]

Therefore, substituting the expression \eqref{eq:v1} into $d_yF((V_i^d(x'))$ in the  above inequality and using the fact that  $d_{x'}F(V_i^d(x'))=e_i$ and $G(x',y) \in \Ker d_y(f_{|X_\alpha})$, we obtain:  
\[ \|\sum_{j=1}^{m}b_j(x',y)e_j+d_yF\left(P(x',y) \right)+d_yF(S(x',y))- e_i\|\leq C\|x-y\|\|V_i^d(x')\|.\] 

Then 
\[ \|\sum_{j=1}^{m}b_j(x',y)e_j-e_i \|-\left\|d_yF\left(P(x',y)\right)\right\|-\|d_yF(S(x',y))\|\leq C\|x'-y\|\|V_i^d(x')\|.\]

We put $M:=\sup_{z\in U_d \cap W_x}\|V_i^d(z)\|$ and $N:={\sup_{z\in W_x}}\|d_zF\|$. Then we get
\begin{equation}\label{eq:last}
\|\sum_{j=1}^{m}b_j(x',y)e_j-e_i \|\leq C(2N+M)\|x'-y\|.
\end{equation} 

If $D={\sup_{z\in W_x\cap X_\alpha}}\|V_{i_\alpha}(z)\|$, it follows by \eqref{eq:defV}, \eqref{eq:v1} and \eqref{eq:last} that
\begin{align*}
\|{V}_i^{d+1}(y)-{V}_i^{d+1}(x')\|< & \, \|\sum_{j=1}^{m}b_j(x',y)V_{j_\alpha}(y) -V_{i_\alpha}(y) \|   +  \| G_{\alpha}(y) - G(x',y) \|\\
&  + \| P(x',y) \|  +  \|S(x',y)\|\\
& < (CD(2N+M) + 3C) \|x'-y\|. 
\end{align*}

Therefore  ${V}_i^{d+1}$ is rugose vector field on $U^{d+}$. Now, it  follows by definition of  ${V}_i^{d+1}$  that 
$ d_x f({V}_i^{d+1}(x))  =e_i$ and $ \langle \nabla \rho(x), {V}_i^{d+1}(x)\rangle =0,$ 
for any $x\in U_{d+1}$. The proof of the Lemma \ref{l:rug-mod} is complete. 
\end{proof}

Continuing inductively we get $m$ rugose vector fields $V_1^{\mathrm{dim}X}, \dots,V_m^{\mathrm{dim}X}$  defined on $U_{\mathrm{dim}X}$, such that for each $i=1,\dots,m$ the following two conditions hold: 
\begin{align}
  \label{eq:1}  d_x f({V}_i^{\mathrm{dim}X}(x)) & =e_i, \\
  \label{eq:2}  \ \langle \nabla \rho(x), {V}_i^{\mathrm{dim}X}(x)\rangle & =0.
\end{align} 

Now, let $\zeta_{x_0}(t)$ be a integral curve of the vector field $V_i^{\dim X}$ starting from  $x_0 \in  (f^{-1}(B) \cap U \cap f_i^{-1}(0))$. Then,  

\begin{enumerate}
\item[a)] the curve $\zeta_{x_0}(t)$ preserves the stratification since $V_i^{\dim X}$ is a rugose vector field. 

\item[b)] it follows by \eqref{eq:1} that $f(\zeta_{x_0}(t)) = te_i$  and by \eqref{eq:2}   that $\|\zeta_{x_0}(t)\| = \|x_0\|$. 
\end{enumerate}

It follows that  the domain of existence of $ \zeta_{x_0}(t)$  is the interval $]-1,1[$ and that  $\zeta_{x_0}(t) \in (f^{-1}(B) \cap U )$, for any $t \in ]-1,1[$. 

Now we can construct the diffeomorphism   trivializing the map
$f: f^{-1}(B) \cap U \to B$ 
by integrating the vector fields $V_1^{\dim X},\dots,V_m^{\dim X}$.
\end{proof}

\begin{proof}[Proof of Theorem \ref{t:main}]

 We will use the vector fields $V_1^{\dim X},\dots,V_m^{\dim X}$ constructed in the proof of Theorem \ref{t:main-ity}. Let $R$ as in the proof of Theorem \ref{t:main-ity} and let $R_1, R_2$ fixed such that $R< R_1< R_2$. 
 
From Proposition \ref{p:lift}, there exist $m$  rugose vector fields  $W_1, \ldots, W_m$ on $f^{-1}(B)$ lifting $\frac{\partial}{\partial 1}, \ldots,  \frac{\partial}{\partial m},$ respectively. 

It follows by \cite[Lemma 2.22]{Lee} that there exists a $C^{\infty}$ function $\varphi: \bR^n\to [0,1]$ such that $\varphi^{-1}(1)= \bar{B}_{R_1}$ and $\varphi^{-1}(0)= \bR^n \setminus {B_{R_2}}$. 

For $i=1, \ldots, m$, we define on $f^{-1}(U)$ the vector fields \[H_i(x):=\varphi(x) W_i(x) + (1-\varphi(x))V^{\dim X}_i(x).\] Then, the restriction of $H_i$ 
on each stratum of $\mathcal{W}$ is a smooth vector field and 
\begin{align}\label{eq:f-Hi}
d_xf(H_i(x))&=d_xf(\varphi(x) W_i(x)+(1-\varphi(x))V_i^{\mathrm{dim}X}(x))\\
\nonumber &=\varphi(x)d_xf(W_i(x))+(1-\varphi(x))d_xf(V_i^{\mathrm{dim}X}(x))\\
\nonumber &=\varphi(x)e_i+ (1-\varphi(x))e_i = e_i,   
\end{align}
for any $i=1, \ldots, m$. Moreover, we have:  

\begin{lemma}
$H_i$ is a rugose vector field  for any $i=1, \dots,m$.  
\end{lemma}
\begin{proof}
Let $x\in f^{-1}(B)$ and let $ X_\beta$ be the stratum of $\mathcal{W}$ containing $x$. Since $W_i$ and $V_i^{\mathrm{dim}X}$ are rugose vector fields, we need to check the rugosity of $H_i$ on $f^{-1}(B) \cap (\overline{B}_{R_2}\setminus B_{R_1})$. Thus, we may assume $x \in f^{-1}(B) \cap (\overline{B}_{R_2}\setminus B_{R_1})$. 

It follows that there exist $C_1, C_2>0$ and a neighborhood $\Omega_x$ of $x$ such that for any  $x'\in X_\beta\cap \Omega_x\cap f^{-1}(B)$ and   $y\in X\cap\Omega_x\cap f^{-1}(B)$, we have:
\begin{eqnarray*}
       \|W_i(y)-W_i(x')\|\leq C_1\|y-x'\|, \\
    \|V_i^{\mathrm{dim}X}(y)-V_i^{\mathrm{dim}X}(x')\|\leq C_2\|y-x'\|.
\end{eqnarray*}

It follows that 
 \begin{align*} 
 \|H_i(y)-H_i(x')\|=&\|\varphi(y)W_i(y) + (1-\varphi(y))V_i^{\mathrm{dim}X}(y)-\varphi(x')W_i(x')- (1-\varphi(x'))V_i^{\mathrm{dim}X}(x')\| \\ 
\leq & \, \|\varphi(y)(W_i(y)-W_i(x'))\| +  \| (1-\varphi(x'))(V^{\dim X}_i(y) - V^{\dim X}_i(x')) \|  \\ 
     +& \,   \|(\varphi(y)- \varphi(x'))W_i(x')\| + \| (1-\varphi(y) -(1-\varphi(x'))) V^{\dim X}_i(y)  \|.    
 \end{align*}

We have:

• $1-\varphi$ is a smooth function and so it is  locally Lipschitz.  Without loss of generality,
assume that $1-\varphi$ is Lipschitz on $\Omega_x$ with constant $C_3$. Then, 
$|(1-\varphi(y))-(1-\varphi(x'))|\leq C_3\|y-x'\|$.

•   $\varphi$ is locally Lipschitz. So, we 
assume that  $\varphi$ is Lipschitz on $\Omega_x$ with constant $C_4$. Then
$|\varphi(y)-\varphi(x')|\leq C_4\|y-x'\|$.

• By the continuity of $W_i$ and $V_i^{\mathrm{dim}X}$, we have that these vector fields are bounded in $\Omega_x$. Thus we take $C_5=\sup_{z\in\Omega_x}\|V_i^{\mathrm{dim}X}(z)\|$ and $C_6=\sup_{z\in\Omega_x}\|W_i(z)\|$. 

Then, 
\[\|H_i(y)-H_i(x')\|\leq (C_1+ C_2 + C_4C_6 + C_3C_5)\|y-x'\|, \]
which shows that  $H_i$ is a rugose vector field. 
\end{proof}  
 Let us return to the proof of the theorem. Let $\xi_{x_0}(t)$ be a integral curve of the vector field $H_i$ starting from  $x_0 \in  (f^{-1}(B) \cap f_i^{-1}(0))$. Then,  

\begin{enumerate}
\item[a)] the curve $\xi_{x_0}(t)$ preserves the stratification since $H_i$ is a rugose vector field (see for instance \cite[Proposition 4.8]{Ve}). 

\item[b)] it follows by \eqref{eq:f-Hi} that $f(\xi_{x_0}(t)) = te_i$. 

\item[c)] by definition of $H_i$, we have $H_i(x) = V_i(x)$, for any $\|x\|> R_2$. Then, $\|\xi_{x_0}(t)\|$ does not converges to infinity as $t$ goes to $\tau$, $|\tau| < 1$. 
\end{enumerate}

It follows  that  the domain of existence of $ \xi_{x_0}(t)$  is the interval $]-1,1[$ and that  $\xi_{x_0}(t) \in (f^{-1}(B))$, for any $t \in ]-1,1[$. 

Now we can construct the diffeomorphism trivializing the map 
$f: f^{-1}(B) \to B$ 
by integrating the vector fields $H_1, \ldots, H_m$.
\end{proof}

\section*{Acknowledgments} 

The first author was partially supported by  Fapemig-Brazil Grant APQ-02085-21 and CNPq-Brazil grant 301631/2022-0. The second author was partially supported by the CAPES Grant PROEX-12113190/D. The authors are grateful to Professor Raimundo Nonato Araújo dos Santos for many helpful discussions.

\end{document}